\documentclass{amsart}[12pt]
\def\doctype{}

\usepackage{latexsym,amssymb}
\usepackage{fancyhdr}
\usepackage{hyperref}

\newcommand\lam{\lambda}

\newcommand{\cA}{\mathcal{A}}
\newcommand{\cB}{\mathcal{B}}

\newcommand\Q{\mathbb{Q}}
\newcommand\Z{\mathbb{Z}}
\newcommand\R{\mathbb{R}}
\newcommand\F{\mathbb{F}}

\newcommand{\comment}[1]{}

\numberwithin{equation}{section}

\fancypagestyle{titlepage}{
\fancyhead[R]{\doctype}
\fancyhead[CL]{}
\cfoot{\vspace{5pt} \thepage}
}

\let\oldsection\section
\newcommand\boldsection[1]{\oldsection{\bf #1}}
\newcommand\starsection[1]{\oldsection*{\bf #1}}
\makeatletter
\renewcommand\section{\@ifstar\starsection\boldsection}
\makeatother

\newtheoremstyle{theorem}
  {12pt}		  % space above
  {0pt}  % space below
  {\sl}  % bofy font
  {\parindent}     % ident - empty=no indent,  \parindent= paragraph indent
  {\bf}  % thm head font
  {. }    % punctuation after thm head
  { }    % space after thm head: `` ``=normal \newline=linebreak
  {}     % thm head specification
\theoremstyle{theorem}
\newtheorem{thm}{Theorem}[section]  % 1st argument is your name for it
\newtheorem{lemma}[thm]{Lemma}     % 2nd argument is what is printed
\newtheorem{cor}[thm]{Corollary}

\newtheorem{prop}[thm]{Proposition}

\newtheoremstyle{definition}
  {12pt}		  % space above
  {0pt}  % space below
  {}  % bofy font
  {\parindent}     % ident - empty=no indent,  \parindent= paragraph indent
  {\bf}  % thm head font
  {. }    % punctuation after thm head
  { }    % space after thm head: `` ``=normal \newline=linebreak
  {}     % thm head specification
\theoremstyle{definition}

\newcommand\rk{{\sc Remark.} }

\renewcommand{\proofname}{Proof}

\makeatletter
\renewenvironment{proof}[1][\proofname]{\par
  \pushQED{\qed}%
  \normalfont \partopsep=\z@skip \topsep=\z@skip
  \trivlist
  \item[\hskip\labelsep
        \scshape
    #1\@addpunct{.}]\ignorespaces
}{%
  \popQED\endtrivlist\@endpefalse
}
\makeatother

\makeatletter
\renewcommand*\@maketitle{%
  \normalfont\normalsize
  \@adminfootnotes
  \@mkboth{\@nx\shortauthors}{\@nx\shorttitle}%
  \global\topskip42\p@\relax % 5.5pc   "   "   "     "     "
  \@settitle
  \ifx\@empty\authors \else {\vskip 1em
\vtop{\centering\shortauthors\@@par}} \fi
  \ifx\@empty\@date \else {\vskip 1em \vtop{\centering\@date\@@par}}\fi % MY CHANGE
  \ifx\@empty\@dedicatory
  \else
    \baselineskip18\p@
    \vtop{\centering{\footnotesize\itshape\@dedicatory\@@par}%
      \global\dimen@i\prevdepth}\prevdepth\dimen@i
  \fi
  \@setabstract
  \normalsize
  \if@titlepage
    \newpage
  \else
    \dimen@34\p@ \advance\dimen@-\baselineskip
    \vskip\dimen@\relax
  \fi
} % end \@maketitle
\renewcommand*\@adminfootnotes{%
  \let\@makefnmark\relax  \let\@thefnmark\relax
  \ifx\@empty\@subjclass\else \@footnotetext{\@setsubjclass}\fi
  \ifx\@empty\@keywords\else \@footnotetext{\@setkeywords}\fi
  \ifx\@empty\thankses\else \@footnotetext{%
    \def\par{\let\par\@par}\@setthanks}%
  \fi
\thispagestyle{titlepage}
}
\makeatother

\begin{document}

\title{\large Threefold triple systems with nonsingular $N_2$}

\author{Peter J.~Dukes}
\address{\rm Peter J.~ Dukes:
Mathematics and Statistics,
University of Victoria, Victoria, Canada
}
\email{dukes@uvic.ca}

\author{Kseniya Garaschuk}
\address{\rm Kseniya Garaschuk:
Mathematics and Statistics,
University of Victoria, Victoria, Canada
}
\email{kgarasch@uvic.ca}

\thanks{Research of Peter Dukes is supported by NSERC grant number 312595--2010}

\date{\today}

\begin{abstract}
There are various results connecting ranks of incidence matrices of graphs and hypergraphs with their combinatorial structure.  Here, we consider the generalized incidence matrix $N_2$ (defined by inclusion of pairs in edges) for one natural class of hypergraphs: the triple systems with index three.  Such systems with nonsingular $N_2$ (over the rationals) appear to be quite rare, yet they can be constructed with PBD closure.  In fact, a range of ranks near $\binom{v}{2}$ is obtained for large orders $v$. 
\end{abstract}

\maketitle
\hrule

\section{Introduction}

We consider hypergraphs with the possibility of repeated edges.  Let $v$ and $\lam$ be positive integers, and suppose $K \subset \Z_{\ge 2}:= \{2,3,4,\dots\}$.
A \emph{pairwise balanced design} PBD$_\lam(v,K)$ is a hypergraph $(V,\cB)$ with $v$ vertices, edge sizes belonging to $K$, and such that
\begin{itemize}
\item
any two distinct vertices in $V$ appear together in exactly $\lam$ edges.
\end{itemize}
In this context, vertices are also called \emph{points} and edges are normally called \emph{blocks}.  The parameter $\lam$ is the \emph{index}; often it is taken to be 1 and suppressed from the notation.  We remark that $K$ could contain unused block sizes.

There are numerical constraints on $v$ given $\lam$ and $K$.  
An easy double-counting argument on pairs of points leads to the \emph{global condition}
\begin{equation}
\label{global}
\lam v(v-1) \equiv 0 \pmod{\beta(K)},
\end{equation}
where $\beta(K):=\gcd\{k(k-1): k \in K\}$.  Similarly, counting incidences with any specific point leads to the \emph{local condition}
\begin{equation}
\label{local}
\lam(v-1) \equiv 0 \pmod{\alpha(K)},
\end{equation}
where $\alpha(K):=\gcd\{k-1: k \in K\}$.  Wilson's theory, \cite{RMW2}, asserts that (\ref{global}) and (\ref{local}) are sufficient for large $v$. 

In the case $K=\{3\}$, we obtain a ($\lam$-fold) \emph{triple system} or TS$_\lam(v)$.  When $\lam=1$ we have a \emph{Steiner triple system} and it is well known that these exist for all $v \equiv 1,3 \pmod{6}$.  In this article we are especially interested in the case $\lam=3$.  The divisibility conditions (\ref{global}) and (\ref{local}) simply reduce to $v$ being odd.  There are $3v(v-1)/6 = \binom{v}{2}$ blocks.  For a comprehensive reference on triple systems, the reader is referred to Colbourn and Rosa's book \cite{CR}.

Given any hypergraph $H=(V,E)$, we may define its \emph{incidence matrix} $N=N(H)$ as the zero-one inclusion matrix of points versus edges. That is, $N$ has rows indexed by $V$, columns indexed by $E$, and where, for $x \in V$, $e \in E$,
$$N(x,e) = \begin{cases}
1 & \text{if } x \in e; \\
0 & \text{otherwise}.
\end{cases}$$

Linear algebraic properties of incidence matrices have received a lot of attention.  Especially interesting are connections with the underlying combinatorial structure.  We give two classical examples.  First, in the case of ordinary graphs, in which $E \subseteq \binom{V}{2}$, it is known \cite{VN} that $N$ has full rank (over $\R$) if and only if every connected component is non-bipartite.  As a different example, the rank of a Steiner triple system over the binary field $\F_2$ is connected in \cite{DHV} with its `projective dimension'.  This measures the length of the lattice of largest possible proper subsystems.  

Let $s$ be a positive integer. The \emph{higher incidence matrix} $N_s$ has a similar definition, but where rows are indexed by $\binom{V}{s}$ (the $s$-subsets of vertices), columns are again indexed by blocks, and entries are defined by inclusion.  That is, for $S \subseteq V$, $|S|=s$, and $e \in E$, we have  
$$N_s(S,e) = \begin{cases}
1 & \text{if } S \subseteq e; \\
0 & \text{otherwise}.
\end{cases}$$
Higher incidence matrices were used by Ray-Chaudhuri and Wilson in \cite{Ontdes} to extend Fisher's inequality to designs of `higher strength'.  In a little more detail, suppose we have a system $(V,\cB)$ of $v$ points, blocks of a fixed size $k$, and every $t$-subset of points belongs to exactly $\lam$ blocks.  These are sometimes denoted S$_\lam(t,k,v)$.  Suppose further that $t$ is even, say $t=2s$, and $v \ge k+s$.  Then the conclusion is that $|\cB| \ge \binom{v}{s}$, and it comes with a strong structural condition for equality.  The matrix $N_s$ plays a key role in the proof.  Incidentally, a new result of Keevash in \cite{Keevash} proves that, for large $v$, the divisibility conditions $\binom{k-i}{t-i} \mid \binom{v-i}{t-i}$ for $i=0,\dots,t$ (which are the analogs of (\ref{global}-\ref{local})) suffice for the existence of S$_\lam(t,k,v)$.

Returning to pairwise balanced designs, higher incidence matrices are of limited use when $\lam=1$. In this case, the matrix $N_2$ is only slightly interesting; each of its rows has exactly one nonzero entry.  The matrix $N_k$ is just, under a reordering of rows, the identity matrix on top of the zero matrix.  In between, $N_s$ for $2<s<k$ has many zero rows and not much structure.

We would like to consider $N_2$ for what is perhaps the first natural case: threefold triple systems TS$_3(v)$.  For such designs, $N_2$ is square of order $\binom{v}{2}$.  In general, we observe that the property of a design having full rank $N_2$ is `PBD-closed'.  From this and some small designs, we have the following main result.

\begin{thm}
\label{main}
There exists a TS$_3(v)$ with $N_2$ nonsingular over $\R$ for all odd $v \ge 5$ except possibly for $v \in E_{579}:= \{v: v \equiv 1 \pmod{2}, v \ge 5,~\text{and} \not\exists~\text{PBD}(v,\{5,7,9\})\}$. 
\end{thm}
It is known (see \cite{Quint} and the summary table entry at \cite{Handbook}, page 252) that
$$E_{579} \subseteq \{11..19, 23, 27..33, 39, 43, 51, 59, 71, 75, 83, 87, 95, 99, 107, 111, 113, 115, 119, 139, 179\},$$
and therefore Theorem~\ref{main} settles the existence question for all but a finite set of values $v$.

The next section sets up and completes the proof.  Then, we conclude with a short discussion of some related topics, including a brief look at such ranks in characteristic $p$.

\section{PBD Closure and Proof of the Main Result}

To prove Theorem~\ref{main}, we first observe that having square nonsingular $N_2$ is a `PBD-closed' property.

\begin{lemma}
\label{lem1}
Suppose there exists a PBD$(v,L)$ and, for each $u \in L$, there exists a PBD$_\lam(u,K)$ having $N_2$ square and full rank over $\F$.  Then there exists a PBD$_\lam(v,K)$ having $N_2$ square and full rank over $\F$.
\end{lemma}

\begin{proof}
Suppose our PBD$(v,L)$ is $(V,\cA)$.  Construct a PBD$_\lam(v,K)$ with points $V$ and block collection 
\begin{equation}
\label{breaking-blocks}
\cB = \bigcup_{U \in \cA} \cB[U],
\end{equation} 
where $\cB[U]$ denotes the blocks of a PBD$_\lam(|U|,K)$ on $U$ having full rank $N_2$.  (Note (\ref{breaking-blocks}) should be interpreted as a formal sum or `multiset union'.)  It is clear that $(V,\cB)$ is a PBD$_\lam(v,K)$.  Consider its incidence matrix $N_2(\cB)$.   If columns are ordered respecting some ordering $U_1,U_2,\dots$ of $\cA$ and the union in (\ref{breaking-blocks}), and rows are ordered respecting $\binom{U_1}{2},\binom{U_2}{2},\dots$, then we obtain a block-diagonal structure
$$N_2(\cB) = N_2(\cB[U_1]) \oplus N_2(\cB[U_1]) \oplus \dots.$$
Since each block is nonsingular, so is $N_2(\cB)$. 
\end{proof}

To clarify, we are working in characteristic zero (rank computed over $\Q$) throughout the remainder of the section.

\begin{lemma}
\label{lem2}
For $v=5,7,9$, there exists a TS$_\lam(v)$ having nonsingular $N_2$.
\end{lemma}

\begin{proof}
The unique TS$_3(5)$ is just the complete design $\binom{[5]}{3}$.  Accordingly, for this design, we have $N_2 N_2^\top = 3I+A$, where $A$ is the adjacency matrix of the line graph of $K_5$ (or complement of the Petersen graph).  Since $A$ is known to have eigenvalues $(-2)^5$, $1^4$, $6^1$, it follows that $N_2$ has full rank.  Examples for $v=7,9$ are given below as a list of blocks on $\{0,\dots,v-1\}$.
\begin{align*}
&\{0,1,2\},\{0,1,3\},\{0,1,4\},\{0,2,3\},\{0,2,5\},\{0,3,6\},\{0,4,5\},\\
v=7:  \hspace{1.5cm} &\{0,4,6\},\{0,5,6\},\{1,2,4\},\{1,2,6\},\{1,3,5\},\{1,3,6\},\{1,4,5\},\\
&\{1,5,6\},\{2,3,4\},\{2,3,5\},\{2,4,6\},\{2,5,6\},\{3,4,5\},\{3,4,6\}.
\end{align*}
\begin{align*}
&\{0,1,2\},\{0,1,3\},\{0,1,4\},\{0,2,3\},\{0,2,5\},\{0,3,6\},\{0,4,6\},\{0,4,7\},\{0,5,7\},\\
v=9: \hspace{.5cm} &\{0,5,8\},\{0,6,8\},\{0,7,8\},\{1,2,4\},\{1,2,5\},\{1,3,6\},\{1,3,8\},\{1,4,7\},\{1,5,6\},\\
&\{1,5,8\},\{1,6,7\},\{1,7,8\},\{2,3,4\},\{2,3,7\},\{2,4,8\},\{2,5,6\},\{2,6,7\},\{2,6,8\},\\
&\{2,7,8\},\{3,4,5\},\{3,4,8\},\{3,5,7\},\{3,5,8\},\{3,6,7\},\{4,5,6\},\{4,5,7\},\{4,6,8\}.
\end{align*}
It is straightforward to confirm these have full rank $N_2$; for example, the following sage code can be used given \verb|v| and a list of sets \verb|B| as above. \qedhere
\begin{verbatim}
T = Set(range(v)).subsets(2)
N2 = matrix(QQ,binomial(v,2))
for i in range(binomial(v,2)):
    for j in range(binomial(v,2)):
        if Set(T[i]).issubset(B[j]):
            N2[i,j]+=1
N2.rank()
\end{verbatim}
\end{proof}

\rk
Of the ten non-isomorphic TS$_3(7)$, exactly one has nonsingular $N_2$.  Of the 22521 TS$_3(9)$, exactly 27 have nonsingular $N_2$. 

The proof of our main result is now an easy combination of the preceding lemmas.

\begin{proof}[Proof of Theorem~\ref{main}]
Take a PBD$(v,\{5,7,9\})$ and replace its blocks as in Lemma~\ref{lem1} by TS$_3(u)$ for $u=5,7,9$ having nonsingular $N_2$, the latter existing by Lemma~\ref{lem2}.  The result is a TS$_3(v)$ having nonsingular $N_2$, as desired.
\end{proof}

\section{Discussion}

A (pairwise) \emph{trade} is a 2-edge-colored hypergraph $(T,\cA_1,\cA_2)$ such that each color class $\cA_i$ covers, counting multiplicity, the same pairs in $\binom{T}{2}$.  A nontrivial example is the `quadrilateral' 
$$\{u,v,a\},\{x,y,a\},\{u,x,b\},\{v,y,b\}$$
together with its image under permuting $a,b$.  Suppose a (multi-)hypergraph $H=(V,\cB)$ contains a trade $(T,\cA_1,\cA_2)$ with $T \subseteq V$ and $\cA_1$, $\cA_2$ as different (multiset) subsets of $\cB$.  Then the trade induces a $\{\pm 1, 0\}$-vector in the kernel of $N_2(H)$.  It follows that some design has $N_2$ of full (column) rank only if it is `trade-free', and in particular, has no repeated blocks.  Accordingly, we have the following direct consequence of Theorem~\ref{main}.

\begin{cor}
There exist trade-free TS$_3(v)$ for all odd integers $v \ge 5$, $v \not\in E_{579}$.
\end{cor}

It may be of interest to compute the set of all possible ranks of $N_2$ over TS$_3(v)$ for a fixed $v$.
When $v \equiv 1,3 \pmod{6}$, one such TS$_3(v)$ comes from three copies of a Steiner triple system, which has rank $\frac{1}{3} \binom{v}{3}$.  It is clear that this is the minimum possible rank.  In the case $v=7$, the complete list of ranks (with repetition) is
$$7,10,12,13,13,15,15,16,18,21.$$
For $v=9$ we compute the list of distinct ranks as
$12, 17, 19, \dots, 36$.
Now, a result of Colbourn and R\H{o}dl in \cite{percentages} guarantees the existence of a PBD$(v,\{5,7,9\}$ for large odd $v$ with many blocks of size 9.  It follows with a similar argument as in the proof of Theorem~\ref{main} that all ranks in the interval $[c\binom{v}{2},\binom{v}{2}]$ are realizable for $c \approx \frac{19}{36}$ and large $v$.  

Finally, we briefly consider $p$-ranks (that is, over $\F_p$, the field of order $p$).  Here are two easy facts.

\begin{prop}
The $2$-rank of $N_2$ for a TS$_3(v)$ is at most $\binom{v-1}{2}$.
\end{prop}

\begin{proof}
Consider the $v-1$ pairs incident with some point, say $x$.  Every block intersects either zero or two such pairs, and hence the corresponding vector in $\R^{\binom{V}{2}}$ lies in the left kernel of $N_2$ over $\F_2$.  There are $v-1$ such independent relations over $\F_2$, and therefore the kernel has dimension at least $v-1$.
\end{proof}

\begin{prop}
The $3$-rank of $N_2$ for a TS$_3(v)$ is at most $\binom{v}{2}-1$.
\end{prop}

\begin{proof}
Observe that $N_2 N_2^\top$ has constant rowsum equal to 9.  So the all-ones vector is in the kernel of $N_2 N_2^\top$ over $\F_3$.
\end{proof}

In our searches for $v=7,9$, we found that both of the above bounds can be met with equality.  Also, it appears likely that, for our problem, $\Q$-rank always agrees with $p$-rank for primes $p>3$.  We presently see no easy argument to confirm this.

\section*{Acknowledgements}

The authors would like to thank Patric R.J.~\"{O}sterg\r{a}rd for providing a data file of all TS$_3(9)$ up to isomorphism.  We would also like to acknowledge Felix Goldberg's question on MathOverflow (and Yuichiro Fujiwara's thoughtful answer) at \url{http://mathoverflow.net/questions/151702/}
which confirmed our belief that ranks of higher incidence matrices could be of interest to the broader community.


\begin{thebibliography}{99}

\bibitem{Quint}
F.E.~Bennett, C.J.~Colbourn, and R.C.~Mullin, Quintessential pairwise
balanced designs. \emph{J. Stat. Plann. Infer.} 72 (1998), 15--66.

\bibitem{Handbook}
C.J.~Colbourn \and J.H.~Dinitz, eds., {\em The CRC Handbook of Combinatorial 
Designs}, 2nd edition, CRC Press, Inc., 2006.

\bibitem{percentages}
C.J.~Colbourn and V.~R\H{o}dl, Percentages in pairwise balanced designs. 
\emph{Discrete Math.} 77 (1989), 57--63.

\bibitem{CR}
C.J.~Colbourn and A.~Rosa, {\em Triple Systems}, Oxford Univ. Press, 1999.

\bibitem{DHV}
J.~Doyen, X.~Hubaut and M.~Vandensavel, Ranks of incidence matrices of Steiner triple systems.  \emph{Math. Z.} 163 (1978), 251--259.

\bibitem{Keevash}
P.~Keevash, The existence of designs, arXiv preprint \url{http://arxiv.org/pdf/1401.3665v1.pdf}, 2014.

\bibitem{Ontdes}
D.K.~Ray-Chaudhuri and R.M.~Wilson, On $t$-designs. Osaka J. Math. 12 (1975), 737--744. 

\bibitem{VN}
C.~Van Nuffelen, On the incidence matrix of a graph. \emph{IEEE Trans. Circuits and Systems} 9 (1976), 572. 

\bibitem{RMW2}
R.M.~Wilson, An existence theory for pairwise balanced designs III: Proof of the
existence conjectures.  {\em J. Combin. Theory Ser. A}
18 (1975), 71--79.


\end{thebibliography}
\end{document}